\documentclass{article}
\usepackage{generic}
\usepackage{cite}
\usepackage{graphicx}
\usepackage{multicol}
\usepackage{textcomp}
\usepackage{algorithm, algorithmicx}
\usepackage{amssymb}
\usepackage{mathrsfs}
\usepackage{enumerate}
\usepackage{caption2}
\usepackage{float}
\usepackage{epstopdf}
\usepackage{amssymb, amsfonts}
\usepackage{amsmath}
\renewcommand\thesection{\Roman{section}}

\newtheorem{theorem}{Theorem}
\newtheorem{lemma}{Lemma}
\newtheorem{assumption}{Assumption}
\newtheorem{remark}{Remark}
\newtheorem{proof}{Proof}

\usepackage[english]{babel}

\usepackage[letterpaper,top=2cm,bottom=2cm,left=3cm,right=3cm,marginparwidth=1.75cm]{geometry}

\usepackage{amsmath}
\usepackage{graphicx}
\usepackage[colorlinks=true, allcolors=blue]{hyperref}

\title{A Distributed Parallel Optimization Algorithm via Alternating Direction Method of Multipliers}
\author{Ziye Liu, Fanghong Guo, Wei Wang, Xiaoqun Wu
	\thanks{Z. Liu and W. Wang are with the School of Automation Science and Electrical Engineering, Beihang University, Beijing, 100191, China (Email: liuziye@buaa.edu.cn, w.wang@buaa.edu.cn. Corresponding Author: Wei Wang).}
	\thanks{F. Guo is with the Department of Automation, Zhejiang University of Technology, Hangzhou 310032, China (Email: fhguo@zjut.edu.cn).}
	\thanks{X. Wu is with the School of Mathematics and Statistics, Wuhan University, Hubei 430072, China (Email: xqwu@whu.edu.cn)}
}

\begin{document}
	\maketitle
	\section{abstract}
	Alternating Direction Method of Multipliers (ADMM) algorithm has been widely adopted for solving the distributed optimization problem (DOP). 
	In this paper, a new distributed parallel ADMM algorithm is proposed, which allows the agents to update their local states and dual variables in a completely distributed and parallel manner by modifying the existing distributed sequential ADMM. 
	Moreover, the updating rules and storage method for variables are illustrated. 
	It is shown that all the agents can reach a consensus by asymptotically converging to the optimal solution. 
	Besides, the global cost function will converge to the optimal value at a rate of $O(1/k)$. 
	Simulation results on a numerical example are given to show the effectiveness of the proposed algorithm. 
	\section{Introduction}
	
	Optimization in networked multi-agent systems has widespread applications including 
	wireless sensor networks \cite{zhang2015distributed}, 
	internet of things \cite{wu2018game}, 
	distributed privacy preservation \cite{zhang2018admm}, 
	distributed Nash Equilibrium seeking \cite{salehisadaghiani2019distributed,salehisadaghiani2016distributed} and 
	machine learning \cite{cortes1995support}. 
	Such a problem was initially handled in a centralized algorithm \cite{boyd2010distributed}, which means that a centralized fusion is required to collect and process the information of all the agents through an underlying communication network. Clearly, tremendous communication and computation resources need to be consumed to implement the centralized optimization methods \cite{yang2019survey,nedic2018network}. For instance, in linear regression problems, it is often prohibitive to gather all the data from different applications of interest \cite{bazerque2010distributed}. In economic dispatch (ED) problems \cite{guo2017hierarchical,guo2017distributed}, centralized manner cannot satisfy the plug-to-play demand. 
	
	To remove the mentioned limitations, distributed optimization problem (DOP) is considered. A common DOP is illustrated as follows,
	\begin{align} 	\label{problem1}
		\min_{x\in \mathbb{R}^d}\quad \sum_{i=1}^{n} f_{i}(x).
	\end{align}
	where $x\in \mathbb{R}^d$ is the global variable to be optimized, $f_{i} : \mathbb{R}^d\rightarrow \mathbb{R}$ is the local cost function privately known by agent $i$, while unknown by other agents for $j\neq i$. In each agent, an estimate of the optimal solution, which is also known as the state, will be generated by applying an iterative algorithm. Then the states need often to be transmitted among networked agents for the purpose of reaching a consensus as the algorithm runs.
	Finally, the global optimization problem (\ref{problem1}) can be solved through collaborations among different agents. 
	
	To solve the DOP, two main classes of distributed optimization algorithms have been proposed, namely the subgradient-based methods and dual-based methods. 
	In the former class, a subgradient computation needs to be conducted in each iteration \cite{nedic2009distributed}. Each agent updates its state with the local subgradient information and states collected from other agents. The subgradient-based methods have been developed for various scenarios \cite{nedic2010constrained,xu2017convergence,duchi2011dual,xi2017dextra,shi2015extra,pu2020push}. 
	In the dual-based methods, a Lagrangian-related functions need to be constructed in the update steps. Alternating Direction Method of Multipliers (ADMM) is a well-known dual-based method, of which the early works can be traced back to \cite{bauschke1994dykstra,combettes2007douglas,zhang2014asynchronous}. By treating the states as primal variables, ADMM minimizes an augmented Lagrangian function and updates both primal and dual variables alternately. In recent years, ADMM has become a popular way to solve DOPs in various areas. 
	Compared with the subgradient-based methods, ADMM-based methods can achieve faster convergence rate \cite{boyd2010distributed,shi2014linear}, and improved robustness against Gaussian noises\cite{simonetto2014distributed}.
	
	
	It is worth noting that for some early works on ADMM-based algorithms developed to solve the optimization problems in multi-agent systems, the utilization of the centralized fusion units cannot be avoided \cite{boyd2010distributed}. 
	In \cite{deng2017parallel}, an ADMM-based optimization algorithm named Jacobi-Proximal ADMM is proposed. It adds a proximal term for each primal variable and a damping parameter for the update step of dual variable. Though the method shows a satisfactory performance, the dual variables need to be handled in a centralized manner. 
	To avoid the need of centralized data processing, some distributed ADMM-based algorithms have been proposed.
	A sequential ADMM is proposed in \cite{wei2012distributed}, in which each agent updates its states by following a predefined order. Both primal and dual variables are computed in a distributed manner. However, the predefined updating order may cost more computational resources, which makes it not well suited for large-scaled systems. 
	A novel distributed ADMM algorithm with parallel manner is proposed in \cite{yan2020parallel}. It reduplicates dual variables for each edge and doubles the dimension of edge-node incidence matrix to make parallelism amenable, hence the algorithm may need extra storage space.
	
	In this paper, the sequential ADMM algorithm in \cite{wei2012distributed} will be modified by presenting a new distributed parallel ADMM algorithm. Extra terms are added in the primal variable updating rule, and the primal variables of different iterations are adopted to update the dual variable for each agent. Main features of the newly presented algorithm are summarized as follows.
	
	\begin{enumerate}[1)]
		\item The presented algorithm is fully distributed, in the sense that a central unit for data collection and processing is not needed for solving the optimization problem. The update of both primal and dual variables for each agent only depends on its local and neighboring information. 
		
		\item The algorithm is suitable for parallel computing. By modifying the existing algorithm in \cite{wei2012distributed}, each agent updates primal variable in a parallel manner rather than a preset sequential order. Moreover, extra dimensions of variables is not needed compared with \cite{yan2020parallel}. 
		
		\item 

		To guarantee the algorithm convergence, compared to the results in [26], an extra indicator matrix and some carefully-designed terms are added in the primal variables updating rule. Moreover, variational inequality method is adopted to indicate the individual state convergence of each agent. 
		

	\end{enumerate}
	
	The remainder of our paper is organized as follows.
	In Section \ref{section2}, the considered unconstrained DOP is presented and the sequential ADMM algorithm in \cite{wei2012distributed} is reviewed.
	In Section \ref{section3}, the distributed parallel ADMM algorithm for solving distributed optimization problem with parallel manner is developed, and the convergence analysis of our presented algorithm is shown in Section \ref{section4}.
	In Section \ref{section5}, numerical simulation results are provided to validate the effectiveness of the proposed method, followed by the conclusion drawn in Section \ref{section6}. 
	
	\emph{Notation:} 
	For a column vector $x$, $x^{\rm T}$ denotes the transpose of $x$. 
	$\|x\|$ denotes the standard Euclidean norm, $\|x\|^{2} =x^{\rm T}x $. 
	For a matrix $A$, $A^{\rm T}$ denotes the transpose of $A$. 
	$A_{ij}$ denotes the entry located on the $i$th row and $j$th column of matrix $A$. $[A]_{i}$ and $[A]^{j}$ denotes the $i$th column and $j$th row of $A$, respectively. 
	For matrices $A=[A_{ij}] \in\mathbb{R}^{m\times n}$ and $B=[B_{ij}]\in\mathbb{R}^{m\times n} $, $A \pm B = [A_{ij} \pm B_{ij}] $ and 
	$\max\{A,B\} \triangleq [\max\{A_{ij},B_{ij}\}] $. 
	$\partial f(x)$ represents the set of all subgradient of function $f(x)$ at $x$. 
	For a nonempty and finite set $F$, $|F|$ denotes the cardinality of the set.

	\section{Problem Statement}\label{section2}
	In this paper, we shall revise the distributed ADMM algorithm in \cite{wei2012distributed}, such that it can be implemented in a parallel manner. Before we present the details of our proposed algorithm, some necessary preliminaries are provided in this section.
	
	\subsection{Problem Statement}
	In this paper, an unconstrained distributed optimization problem, as shown in (1), will be considered. To be more specific, $n$ agents indexed by $1,2,..., n$ are aimed at solving the optimization problem collaboratively under a fixed undirected graph $\mathcal{G}=\{\mathcal{V}, \mathcal {E}\}$, where $\mathcal{V}=\{1, 2... n\}$ and $\mathcal{E}$ denote the set of nodes and the set of edges connecting two distinct agents, respectively. Suppose that there is at most one edge $e_{ij}$ with a subscript pair $(i,j)$ connecting agents $i$ and $j$. We prescribe that the first subscript $i$ and the second subscript $j$ satisfy $i < j$. If agents 1 and 2 are connected in a given graph, $e_{12}$ represents the sole edge connecting them.
	
	Let us recall the basic idea of solving a distributed optimization problem \cite{boyd2010distributed,wei2012distributed} 
	\begin{align} \label{problem2}
		\min_{x\in \mathbb{R}}\quad \sum_{i=1}^{n} f_{i}\left(x\right)
	\end{align}
	where $x\in \mathbb{R}$ is the global variable to be optimized, $f_{i}:$ $\mathbb{R}\rightarrow \mathbb{R}$ is the local cost function stored in agent $i$.

	The goal of Problem (\ref{problem2}) is to find the minimizer $x^*$, such that the function $\sum_{i=1}^{n} f_{i}(x)$ can be minimized when $x=x^\ast$. To develop a distributed algorithm for solving Problem (\ref{problem2}), this problem should often be reformulated to accommodate the distributed property. Normally, each agent $i$ holds an \emph{estimate} of $x^\ast$, which can be denoted by $x_i$. As the optimization algorithm goes on, each agent $i$, for $i\in \mathcal{V}$, will utilize its local estimate and the information received from its neighbors till current iteration to update its local estimate for the next iteration. After sufficient iterations, the estimates for all agents tend to reach a consensual value. The form of Problem (\ref{problem2}) can thus be changed as follows,
	\begin{align}\label{problem3}
		\min_{x_{i}\in \mathbb{R}} \quad &\sum_{i=1}^{n} f_{i}\left(x_{i}\right), \nonumber \\
		\operatorname{s.t.} \quad &x_i=x_j,\forall e_{ij}\in\mathcal{E}.
	\end{align}
	
	The linear constraint implies the consensus mentioned above, hence Problem (\ref{problem3}) and Problem (\ref{problem2}) have the same solution.
	
	\begin{remark}
		It is worth mentioning that multi-dimensional problems like Problem \ref{problem1} are more general in practice. For simplicity, we consider one-dimensional scenario. Multi-dimensional can be expanded from one-dimensional problem by Kronecker product \cite{wei2012distributed,yan2020parallel}.
	\end{remark}

	\subsection{Distributed Sequential ADMM Algorithm \cite{wei2012distributed}}
	To solve Problem (\ref{problem3}), a distributed sequential ADMM algorithm is presented in \cite{wei2012distributed}. Firstly, the updating rules of agent $i$ in the $k$th iteration are developed as follows, 
	\begin{align}
		x_{i}^{k+1} =& \underset{x_{i}}{\operatorname{argmin}} 
		\Bigg \{ 
		f_{i}\left(x_{i}\right) \nonumber \\
		&+\frac{\rho}{2} \sum_{j \in P_{i}} \left \|x_{j}^{k+1}-x_{i}-\frac{1}{\rho} \lambda_{j i}^{k}\right \|^{2} \nonumber\\
		&+\frac{\rho}{2} \sum_{j \in S_{i}}\left \|x_{i}-x_{j}^{k}-\frac{1}{\rho} \lambda_{i j}^{k}\right\|^{2} 
		\Bigg \}, \label{s1}\\
		\lambda_{ji}^{k+1}=& \lambda_{ji}^{k}-\rho\left(x_{j}^{k+1}-x_{i}^{k+1}\right)
		\label{s2}
	\end{align}
	where $\rho$ is a positive penalty parameter. $\lambda_{ji}$ denotes the dual variable associated with the constraint on edge $e_{ji}$ stored in agent $i$. Define $N_{i}=\{j\in \mathcal{V}|e_{ij}\in \mathcal{E}\}$ as the set of neighbors of agent $i$. Furthermore, the subsets of neighbors $P_{i}=\{j|e_{ji}\in \mathcal{E}, j<i\}$ and $S_{i}=\{j|e_{ij}\in \mathcal{E}, j>i\}$ are called the \emph{predecessors} and \emph{successors} of agent $i$, respectively. Take the sample graph of 4 agents in Fig. \ref{fig1} as an example. $N_2 = \{1,4\}$, $P_2=\{1\}$, and $S_2=\{4\}$. 
	\begin{figure}[H]
		\includegraphics[height=2.3cm]{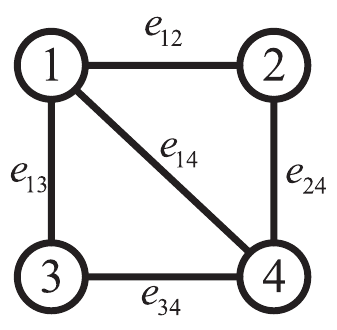}
		\centering
		\caption{Network example with 4 nodes and subscripts of each edge.}\label{fig1}
		\label{f1}
	\end{figure}
	
	The detailed procedure of the distributed sequential ADMM algorithm developed in \cite{wei2012distributed} is provided below.
	\begin{algorithm}[h]
		\caption{\cite{wei2012distributed} Distributed sequantial ADMM algorithm}
		\begin{enumerate}[1:]
			\item \textbf{Inputs:} $x_{i}^{0}, \lambda_{ij}^{0}$ for each $i\in \mathcal{V}$. 
			\item \textbf{for} $k=1, 2... $\textbf{do}
			\item \quad \textbf{for} $i=1,2...n $ \textbf{do}
			\item \qquad Agent $i$ updates $x_{i}^{k+1}$ according to (\ref{s1}).
			\item \qquad Agent $i$ sends $x_{i}^{k+1}$ to its successors. 
			
			\item \quad \textbf{end for}
			\item Each agent $i$ updates its $\lambda_{ji}^{k+1}$ for $j \in P_{i}$ according to (\ref{s2}), and send $\lambda_{ji}^{k+1}$ to its neighbors.
			\item \textbf{end for}
			\item \textbf{Output:} $x_i^k$ $\forall i\in \mathcal{V}$ 
		\end{enumerate}
	\end{algorithm}
	
	\begin{remark} 
		The distributed sequential ADMM algorithm in \cite{wei2012distributed} is developed based on standard ADMM \cite{boyd2010distributed}.
		The distributed feature of the algorithm can be observed from the updating rules (4)-(5). More specifically, only locally available information received from its neighbors is needed to update both primal variable $x_i$ and dual variable $\lambda_{ji}$, with $j<i$, for each agent $i$. In other words, center node is not necessary for data collection and processing to implement the algorithm. 
		
		Nevertheless, the updating process in (\ref{s1}) is calculated in a sequential order from agent $1$ to $n$. Each agent $i$ needs to collect its \emph{predecessors'} information obtained in iteration $k+1$ (i.e. $x_j^{k+1}, j\in P_i$) to update its local state in iteration $k+1$ (i.e. $x_i^{k+1}$). It implies that the updating processes for the entire group should be operated from agent $1$ to agent $n$ in a sequential order, as depicted in Fig. \ref{fig2}. Hence, the sequential ADMM algorithm may not be suitable to solve distributed optimization problems for large scaled multi-agent systems, since the total waiting time will be increased to an unacceptable level when the number of agents is large. 
		
		%
	\end{remark}
	
	\begin{figure}[H]
		\centering
		\includegraphics[height=2.2cm]{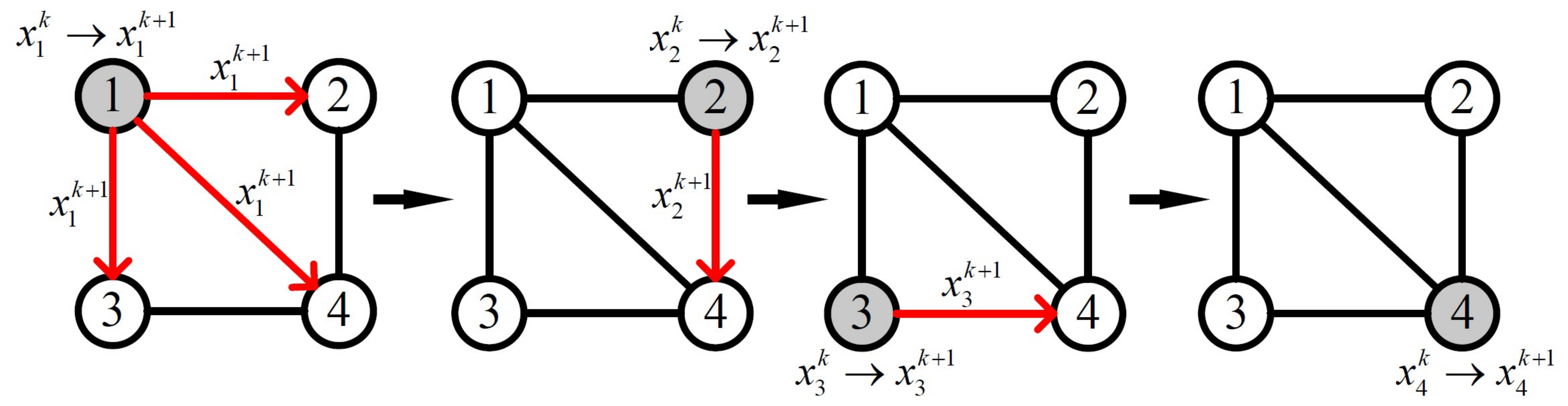}
		\centering
		\caption{$x_i$ updates in a sequential manner. Obviously, only when an agent finishes updating and transmitting procedures of its local state, its successor neighbour can perform its calculations.}\label{fig2}
	\end{figure}
	
	\section{A Distributed Parallel ADMM algorithm}\label{section3}
	
	To remove the limitation of distributed sequential ADMM algorithm \cite{wei2012distributed} as summarized in Section \ref{section2}, a new distributed ADMM algorithm is presented in this section. We first present the updating rules of each agent $i$ in the $k$th iteration:

	
	\begin{align}
		x_{i}^{k+1} 
		=& \underset{x_{i}}{\operatorname{argmin}} \Bigg\{f_{i} (x_{i}) \nonumber\\
		&+\left(\sum_{j\in P_{i}} \lambda _{ji}^{k}
		- \sum_{j\in S_{i}} \lambda _{ij}^{k}\right) x_{i}
		+\frac{\rho}{2}\sum_{j\in N_{i}} \left\|x_{i}-x_{j}^{k} \right\|^{2}\nonumber\\
		&+\left(1+\epsilon_1\right)\rho |P_{i}| \left\|x_{i}-x_{i}^{k} \right\|^{2}
		+\epsilon_2\rho |S_{i}| \left\|x_{i}-x_{i}^{k} \right\|^{2}\nonumber\\
		&-\rho \left(\sum_{j\in P_{i}} x_{j}^{k} 
		- |P_{i}| x_{i}^{k}\right) x_{i} \Bigg\}, \label{MADMM1}\\
		\lambda_{ij}^{k+1}=&\lambda_{ij}^{k} - \rho\left(x_{i}^{k+1}-x_{j}^{k}\right) \label{MADMM2}
	\end{align}
	where $\rho$ is a positive penalty parameter. $\epsilon_1$ and $\epsilon_2$ are constants satisfying $\epsilon_1>0$ and $\epsilon_2>0$. $P_i$ and $S_i$ are defined the same as mentioned above in Section \ref{section2}. $\lambda_{ij}$ denotes the dual variable associated with the constraint on edge $e_{ij}$ stored in agent $i$. Notice that according to the subscript, dual variable $\lambda_{ij}$ is stored in different agents in these two algorithms. 
	The detailed procedure of the newly presented distributed ADMM algorithm is summarized below. 
	
	
	\begin{algorithm}[H]
		\caption{Distributed Parallel ADMM algorithm}
		\begin{enumerate}[1:]
			\item \textbf{Input:} $x_{i}^{0}, \lambda_{ij}^{0}$ for each $i\in \mathcal{V}$. 
			\item \textbf{for} $k=1, 2... $\textbf{do}
			\item \quad Each agent $i$ updates $x_{i}^{k+1}$ according to (\ref{MADMM1}).
			\item \quad Each agent $i$ updates $\lambda_{ij}^{k+1}$ for $j \in S_{i}$ according to (\ref{MADMM2}).
			
			\item \quad Each agent $i$ sends $x_{i}^{k+1}$ and $\lambda_{ij}^{k+1}$ to its neighbors. 
			\item \textbf{end for}
			\item \textbf{Output:} $x_i^k$ $\forall i\in \mathcal{V}$ 
		\end{enumerate}
	\end{algorithm}
	
	%
	
	\begin{remark}
		Note that the definitions of the \emph{predecessors} and \emph{successors} of the distributed sequential ADMM algorithm in \cite{wei2012distributed} and our newly proposed algorithm are the same, mathematically. 

		However, the proposed algorithm alternates the sequential updating manner to a parallel way. As observed from (6), no terms related to $x_j^{k+1}$ are needed to update $x_i^{k+1}$. Hence all the $n$ agents can update both primal and dual variables simultaneously, rather than waiting for its predecessors to complete their updating procedures. For easier understanding the major difference between Algorithms 1 and 2, the parallel updating manner of $x_i^{k+1}$ for the considered group of 4 agents is plotted in Fig. \ref{fig3}.
	\end{remark}
	\begin{figure}[H]
		\centering
		\includegraphics[height=2.3cm]{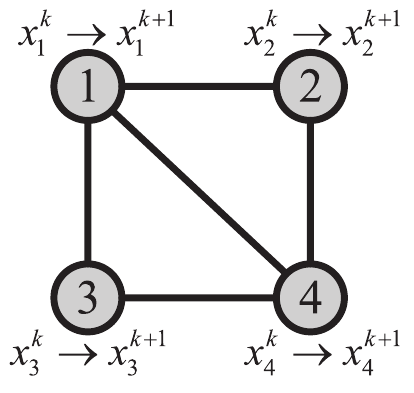}
		\centering
		\caption{$x_i$ updates with parallel manner, agents can simultaneously update states.}\label{fig3}
	\end{figure}
	
	


	\section{Convergence Analysis}\label{section4}	
	In this section, we shall firstly show that all the agents' states $x_i^k$ will converge to a unique optimal solution $\bar x^\ast$ as the number of iterations approaches infinity. Then the convergence rate will be analyzed. 
	
	Denote the global cost function by $F({x})=\sum_{i=1}^{n}f_{i}({x_{i})}$, where $x=[x_1,x_2,...,x_n]^{\rm T}\in \mathbb{R}^n$, and the optimal solution vector by $\bar{x}^*=\bold{1} x^*\in \mathbb{R}^n$, where $\bold{1}$ is the vector that all elements are $1$. Let the vector $\lambda$ be the Lagrange multiplier vector. For clarity, $\lambda$ is stacked by dual variables with subscript numbers arranged in the ascending lexicographical order. 
	Then an edge-node incidence matrix $A=[A_{pq}]\in \mathbb{R}^{m \times n}$ as defined in \cite{wei2012distributed} is reintroduced to describe the existence of edges among any two distinct agents for a given graph. 
	Suppose that $\lambda_{pq}$ is the $r$th element of $\lambda$, which corresponds to the $r$th row of $A$ (denoted by $[A]^{r}$). Note that the elements in $[A]^{r}$ satisfy that $A_{rp} = 1, A_{rq} = -1$ and other entries are $0$. Taking Fig. \ref{fig1} as an example, the vector $\lambda$ is written as $\lambda =[\lambda_{12}, \lambda_{13}, \lambda_{14}, \lambda_{24}, \lambda_{34}]^{\rm T}$. The corresponding edge-incidence matrix $A$ is given below. 
	\begin{align}
		A=\begin{bmatrix}\label{A}
			1 & -1 & 0 & 0\\ 
			1 & 0 & -1 & 0\\ 
			1 & 0 & 0 & -1\\ 
			0 & 1 & 0 & -1\\ 
			0 & 0 & 1 & -1
		\end{bmatrix}
	\end{align}
	%
	
	Thus Problem (\ref{problem3}) can be rewritten in the following compact form,
	\begin{align}\label{problem4}
		\min_{x_{i}\in \mathbb{R}} \quad &\sum_{i=1}^{n} f_{i}\left(x_{i}\right), \nonumber \\
		\operatorname{s.t.} \quad &Ax=0.
	\end{align}
	\begin{remark}
		Note that (\ref{problem2}), (\ref{problem3}) and (\ref{problem4}) are actually the same problem, though in different forms. In this section, we use (\ref{problem4}) uniformly in the following lemmas and theorems.
	\end{remark}
	
	We introduce two additional matrices, i.e. $B=\max\{\bold{0}, A\}$ and $E=A-B$ for further illustration, where $\bold{0}$ represents a matrix with all entries zero. Denote Lagrangian function of this problem by $L(x,\lambda)=F(x)-\lambda^{\rm T}Ax$. Then the following standard assumptions in\cite{wei2012distributed}, indicating the property of cost functions and the existence of the saddle point, are imposed. 
	
	\begin{assumption}
		Each cost function $f_{i}: \mathbb{R}\rightarrow \mathbb{R}$ is closed, proper and convex.
	\end{assumption}
	
	\begin{assumption}
		The Lagrangian function has a saddle point, that is there exists a solution variable pair $\{x^{*}, \lambda^{*}\}$ with
		\begin{align}
			L\left(x^{*}, \lambda\right) \leqslant L\left(x^{*}, \lambda^{*}\right) \leqslant L\left(x, \lambda^{*}\right).
		\end{align}
	\end{assumption}
	
	Before establishing the main theorem, the following lemmas are presented as preliminary results. 

	\begin{lemma}(See \cite{deng2017parallel})
		\label{lemma1}
		Problem (\ref{problem4}) is equivalent to the variational inequality as follows,
		\begin{align}\label{VI}
			F\left(x\right) - F\left(x^{*}\right) 
			+\left(u - u^{*}\right)^{\rm T}\mathscr{F}\left(u^{*}\right) \geqslant 0, \forall u
		\end{align}
		where $u=(x^{\rm T}, \lambda^{\rm T})^{\rm T}\in \mathbb{R}^{n+m}$, $u^{*}=({x^*}^{\rm T}, {\lambda^*}^{\rm T})^{\rm T}\in \mathbb{R}^{n+m}$. $\mathscr{F}(u)=[-[A]_1^{T} \lambda, -[A]_2^{T} \lambda,..., -[A]_{n}^{T}\lambda, Ax]^{T}$ is the solution of the variational inequality (3). 
	\end{lemma}

	It will be seen in Section \ref{section4} that Lemma \ref{lemma1} is essential to complete the proof of optimality.
	
	
	
	Next a basic lemma for convergence analysis is established.
	
	\begin{lemma}\label{lemma3}
		The following inequality holds for all $k$:
		\begin{align}
			&F\left(x\right)-F\left(x^{k+1}\right)+\left(x-x^{k+1}\right)^{\rm T} \nonumber\\
			&\cdot \Bigg\{ -A^{\rm T}\lambda^{k+1}
			+\rho\left(4E^{\rm T}E-A^{\rm T}E\right)\left(x^{k+1}-x^{k}\right) \nonumber\\ 
			& \quad +\rho\left(2\epsilon_1E^{\rm T}E+2\epsilon_2 B^{\rm T}B\right) \left(x^{k+1}-x^{k}\right) \Bigg\}\nonumber\\
			&+\left(x-x^{k+1}\right)^{\rm T}\rho A^{\rm T}Ex^{k} \geqslant 0. 
		\end{align}
	\end{lemma}
	
	\begin{proof}
		We define $g_{i}$ as follows,
		\begin{align}\label{g}
			g_{i}^{k}\left(x_{i}\right)
			=&\left(\sum_{j\in P_{i}} \lambda _{ji}^{k}
			- \sum_{j\in S_{i}} \lambda _{ij}^{k}\right) x_{i}
			+\frac{\rho}{2}\sum_{j\in N_{i}} \left\|x_{i}-x_{j}^{k} \right\|^{2}\nonumber\\
			&+\left(1+\epsilon_1\right)\rho |P_{i}| \left\|x_{i}-x_{i}^{k} \right\|^{2}
			+\epsilon_2\rho |S_{i}| \left\|x_{i}-x_{i}^{k} \right\|^{2}\nonumber\\
			&-\rho \left(\sum_{j\in P_{i}} x_{j}^{k} 
			- |P_{i}| x_{i}^{k}\right) x_{i} 
		\end{align}
		
		Then the gradient of $g_{i}^{k}(x_{i}^{k+1})$ is 
		\begin{flalign}\label{nablag}
			&\nabla g_{i}^{k}\left(x_{i}^{k+1}\right)\nonumber\\
			=& \left(\sum_{j \in P_{i}} \lambda_{ji}^{k}-\sum_{j \in S_{i}} \lambda_{i j}^{k} 
			\right) + \rho\left(|N_i|x_{i}^{k+1} -\sum_{j \in N_{i}}x_{j}^{k}\right) \nonumber\\
			& +\left(2+2\epsilon_1\right)\rho |P_{i}| \left(x_{i}^{k+1}-x_{i}^{k}\right) 
			+ 2\epsilon_2\rho |S_{i}| \left(x_{i}^{k+1}-x_{i}^{k}\right) \nonumber\\
			&-\rho\left(\sum_{j \in P_{i}} x_{j}^{k}-|P_{i}| x_{i}^{k}\right)\nonumber\\
			=&\left(\sum_{j \in P_{i}} \lambda_{ji}^{k+1}-\sum_{j \in S_{i}} \lambda_{ij}^{k+1} \right) \nonumber\\
			&+ \rho\Bigg\{\sum_{j\in P_{i}} \left(x_j^{k+1}-x_j^k\right) 
			+ |P_{i}|\left(x_i^{k+1}-x_i^k\right) \nonumber\\ 
			&\qquad+\left(2+2\epsilon_1\right) |P_{i}| \left(x_{i}^{k+1}-x_{i}^{k}\right) 
			+ 2\epsilon_2 |S_{i}| \left(x_{i}^{k+1}-x_{i}^{k}\right) \nonumber\\
			&\qquad-\left(\sum_{j \in P_{i}} x_{j}^{k}-|P_{i}| x_{i}^{k}\right) \Bigg\}
		\end{flalign}
		where (\ref{MADMM2}) and the relationship $(|N_i|x_i^{k+1}-\sum_{j \in N_{i}}x_{j}^{k}) - (|S_{i}|x_i^{k+1}-\sum_{j \in S_{i}}x_j^k)=|P_{i}|x_i^{k+1}-\sum_{j \in P_{i}}x_j^k$ has been utilized in (\ref{nablag}).
		Then denoting $h_{i}(x_{i}^{k+1})\in \partial f_{i}(x^{k+1})$, we have $h_{i}(x_{i}^{k+1})+\nabla g_{i}^{k}(x_{i}^{k+1})=0$ and $(x_{i}-x_{i}^{k+1})[h_{i}(x_{i}^{k+1})+\nabla g_{i}^{k}(x_{i}^{k+1})]=0$. By recalling the definiation of subgradient:
		$f_{i}(x_{i}) \geqslant f_{i}(x_{i}^{k+1})+(x_{i}-x_{i}^{k+1})^{\rm T}h_{i}(x_{i}^{k+1})$,
		the following inequality is obtained. 
		\begin{align}\label{16}
			f_{i}\left(x_{i}\right)-f_{i}\left(x_{i}^{k+1}\right)+\left(x_{i}-x_{i}^{k+1}\right)^{\rm T} \nabla g_{i}^{k}\left(x_{i}^{k+1}\right) \geqslant 0. 
		\end{align}
		
		By substituting $\nabla g_i^{k}(x_i^{k+1})$ in (\ref{16}), we obtain
		\begin{align}\label{vital}
			&f_{i}\left(x_{i}\right)-f_{i}\left(x_{i}^{k}\right) \nonumber\\
			&+\left(x_{i} - x_{i}^{k+1}\right) ^{\rm T}\left(\sum_{j \in P_{i}} \lambda_{ji}^{k+1}-\sum_{j \in S_{i}} \lambda_{ij}^{k+1}\right) \nonumber\\
			&+\rho\left(x_{i}-x_{i}^{k+1}\right)^{\rm T}\nonumber\\
			& \cdot \Bigg\{\sum_{j\in P_{i}} \left(x_j^{k+1}-x_j^k\right) 
			+|P_{i}|\left(x_i^{k+1}-x_i^k\right) \nonumber\\ 
			&\quad+\left(2+2\epsilon_1\right) |P_{i}| \left(x_{i}^{k+1}-x_{i}^{k}\right) + 2\epsilon_2 |S_{i}| \left(x_{i}^{k+1}-x_{i}^{k}\right) \nonumber\\
			&\quad-\left(\sum_{j \in P_{i}} x_{j}^{k}-|P_{i}| x_{i}^{k}\right) \Bigg\}\geqslant 0.
		\end{align}
		
		The following equations hold for all $k$:
		\begin{flalign}
			&\sum_{i=1}^{n}\left(x_{i}-x_{i}^{k+1}\right)^{\rm T}\left(\sum_{j \in P_{i}} \lambda_{ji}^{k+1}-\sum_{j \in S_{i}} \lambda_{i j}^{k+1}\right)\nonumber\\
			=&\sum_{i=1}^{n}\left(x_{i}-x_{i}^{k+1}\right)^{\rm T}
			\left(-[A]_{i}^{\rm T} \lambda^{k+1}\right)\nonumber\\
			=&\left(x-x^{k+1}\right)^{\rm T} \left(-A^{\rm T} \lambda^{k+1}\right), \label{relation1}
		\end{flalign}
		
		\begin{flalign}
			&\sum_{i=1}^{n}\left(x_{i}-x_{i}^{k+1}\right)^{\rm T}
			\Bigg \{\sum_{j\in P_{i}} \left(x_j^{k+1}-x_j^k\right) 
			+|P_{i}|\left(x_i^{k+1}-x_i^k \right) \Bigg \}\nonumber\\
			=&\left(x-x^{k+1}\right)^{\rm T}\left(-B^{\rm T}E +E^{\rm T} E \right) \left( x^{k+1} - x^{k} \right) \nonumber\\
			=&\left(x-x^{k+1}\right)^{\rm T}\left(2E^{\rm T} E-A^{\rm T} E\right)\left(x^{k+1}-x^{k}\right), \label{relation2}
		\end{flalign}

		\begin{align}
			&\sum_{i=1}^{n}\left(x_{i}-x_{i}^{k+1}\right)^{\rm T}|P_{i}|\left(x_{i}^{k+1}-x_{i}^{k}\right)\nonumber\\
			=&\left(x-x^{k+1}\right)^{\rm T}E^{\rm T}E\left(x^{k+1}-x^{k}\right)\label{relation3}
		\end{align}
		
		\begin{align}
			&\sum_{i=1}^{n}\left(x_{i}-x_{i}^{k+1}\right)^{\rm T}|S_{i}|\left(x_{i}^{k+1}-x_{i}^{k}\right)\nonumber\\
			=&\left(x-x^{k+1}\right)^{\rm T}B^{\rm T}B\left(x^{k+1}-x^{k}\right)\label{relation4}
		\end{align}
		
		\begin{align}		&\sum_{i=1}^{n}\left(x_{i}-x_{i}^{k+1}\right)^{\rm T}\left(\sum _{j\in P_{i}}x_{j}^{k}-|P_{i}|x_{i}^{k}\right)\nonumber\\	
			=&\left(x-x^{k+1}\right)\left(-B^{\rm T} E - E^{\rm T}E\right)x^{k} \nonumber\\
			=&\left(x-x^{k+1}\right)\left(-A^{\rm T} E\right)x^{k},\label{relation5}
		\end{align}
		where Eqns. (\ref{relation2}) and (\ref{relation5}) hold due to the fact that $B = A- E$. By substituting (\ref{relation1})-(\ref{relation4}) in (\ref{vital}) and summing over $1,2,...,n$, we can complete the proof.
	\end{proof}

	The following theorem illustrates the main results obtained with our presented algorithm. 
	
	\begin{theorem}
		The state $x_{i}$ generated by Algorithm 2 in each agent $i$ converges to the optimal solution of Problem (\ref{problem4}), and the global cost function converges to the optimal value with $k \rightarrow \infty$. i.e., 
		\begin{align}
			&\lim _{k \rightarrow \infty}x_{i}^{k}=\bar{x}^{*}, i=1, 2,... n, 
		\end{align}
		where $\bar{x}^{*}$ and $F(x^*)$ are the optimal solution and the optimal value, respectively. 
	\end{theorem}
	
	\begin{proof}
		From (\ref{MADMM2}) and $Ax^{*}=0$, $\lambda^{k+1}=\lambda^{k}-\rho(A-E) x^{k+1}-\rho E x^{k}$ can be derived. By setting $x = x^{*}$, it follows that 
		
		\begin{flalign}\label{relation6}
			&\left(x^{k+1}\right)^{\rm T}A^{\rm T}\left(\lambda^{k+1}-\lambda^{*}\right) \nonumber\\
			=&\frac{1}{2\rho}\left(\left\| \lambda^{k}-\lambda^{*}-\rho Ex^{k} \right\|^{2}-\left\| \lambda^{k+1}-\lambda^{*}-\rho Ex^{k+1} \right\|^{2}\right)\nonumber\\
			&-\frac{\rho}{2}\left\|Ax^{k+1}\right\|^{2} + \rho\left(x^{k+1}\right)^{\rm T}A^{\rm T}Ex^{k+1}.
		\end{flalign}
		
		\begin{flalign}\label{relation7}
			&4\rho\left(x^* - x^{k+1}\right)^{\rm T}E^{\rm T}E\left(x^{k+1}-x^{k}\right)\nonumber\\
			=&-2\rho \left\| E\left(x^{k+1}-x^{k}\right) \right\|^{2} \nonumber\\
			& +2\rho \left(\left\| E\left(x^{k}-x^{*}\right) \right\|^{2}-\left\| E\left( x^{k+1}-x^{*}\right) \right\|^{2}\right), 
		\end{flalign}
		
		\begin{flalign}\label{relation8}
			&\rho\left(x^* - x^{k+1}\right)^{\rm T} 
			\left(2\epsilon_1 E^{\rm T}E +2\epsilon_2B^{\rm T}B \right) \left(x^{k+1}-x^{k}\right)\nonumber\\
			=&	-\epsilon_1\rho \left\| E\left(x^{k+1}-x^{k}\right) \right\|^{2}-\epsilon_2\rho \left\| B\left(x^{k+1}-x^{k}\right) \right\|^{2}\nonumber\\	
			&+ \epsilon_1\rho \left(\left\| E\left(x^{k}-x^{*}\right) \right\|^{2}-\left\| E\left(x^{k+1}-x^{*}\right) \right\|^{2}\right)\nonumber\\
			&+ \epsilon_2\rho \left(\left\| B\left(x^{k}-x^{*}\right) \right\|^{2}-\left\| B\left(x^{k+1}-x^{*}\right) \right\|^{2}\right),
		\end{flalign}
		
		\begin{flalign}\label{relation9}
			&2\rho\left(x^{k+1}\right)^{\rm T}A^{\rm T}E\left(x^{k+1}-x^{k}\right) \nonumber\\
			=&-\frac{\rho}{2}\left\|2E\left(x^{k+1}-x^{k}\right)-Ax^{k+1}\right\|^{2}\nonumber\\ 
			&+\frac{\rho}{2}\left\|Ax^{k+1}\right\|^{2} +2\rho\left\|E\left(x^{k+1}-x^{k}\right)\right\|^{2}.
		\end{flalign}

		Then the following relation holds:
		\begin{align} \label{combining1}
			&\left(x^{k+1}\right)^{\rm T}A^{\rm T}\left(\lambda^{k+1}-\lambda^{*}\right)\nonumber\\ &+\rho\left(x^{*}-x^{k+1}\right)^{\rm T}\left(4E^{\rm T}E-A^{\rm T}E\right) \left(x^{k+1}-x^{k}\right) \nonumber\\
			&+\rho\left(x^{*}-x^{k+1}\right)^{\rm T}\left(\epsilon_1 E^{\rm T}E + \epsilon_2 B^{\rm T}B\right) \left(x^{k+1}-x^{k}\right) \nonumber\\
			&+\rho\left(x^{k+1}\right)^{\rm T}\left(-A^{\rm T}E\right) x^{k}\nonumber\\
			=&-\frac{\rho}{2} \left\| 2E\left( x^{k+1}-x^{k}\right)-Ax^{k+1} \right\|^2 \nonumber\\
			&-\epsilon_1\rho \left\| E\left(x^{k+1}-x^{k}\right) \right\|^{2}-\epsilon_2\rho \left\| B\left(x^{k+1}-x^{k}\right) \right\|^{2}\nonumber\\
			&+\frac{1}{2\rho} \left( \left\| \lambda^{k}-\lambda^{*}-\rho Ex^{k} \right\|^2- 
			\left\| \lambda^{k+1}-\lambda^*-\rho Ex^{k+1} \right\|^2 \right) \nonumber\\
			&+\left(2+\epsilon_1\right)\rho\left(\left\| E\left(x^k-x^*\right) \right\|^2 - \left\| E\left(x^{k+1}-x^{*}\right) \right\|^2 \right)\nonumber\\
			&+ \epsilon_2\rho \left(\left\| B\left(x^{k}-x^{*}\right) \right\|^{2}-\left\| B\left(x^{k+1}-x^{*}\right) \right\|^{2}\right).
		\end{align}
		
		By rearranging the terms in Lemma \ref{lemma3}, the following inequality is established:
		\begin{align}\label{combining2}
			&\left(x^{k+1}\right)^{\rm T}A^{\rm T}\left(\lambda^{k+1}-\lambda^{*}\right)\nonumber\\ &+\rho\left(x^{*}-x^{k+1}\right)^{\rm T}\left(4E^{\rm T}E-A^{\rm T}E\right) \left(x^{k+1}-x^{k}\right) \nonumber\\
			&+\rho\left(x^{*}-x^{k+1}\right)^{\rm T}\left(2\epsilon_1 E^{\rm T}E + 2\epsilon_2 B^{\rm T}B\right) \left(x^{k+1}-x^{k}\right) \nonumber\\
			&+\rho\left(x^{k+1}\right)^{\rm T}\left(-A^{\rm T}E\right) x^{k}\nonumber\\
			\geqslant &F\left(x^{k+1}\right)-\left(x^{k+1}\right)^{\rm T}A^{\rm T}\lambda^{*}-F\left(x^{*}\right) \geqslant 0
		\end{align}
		where in (\ref{combining2}) the relation of the saddle point of Lagrangian function relation $F(x^{*})-\lambda^{\rm T} A x^{*} \leq F(x^{*})-(\lambda^{*})^{\rm T} A x^{*}$ has been used. 
		
		Let $V^{k}=\frac{1}{2\rho} \| \lambda^{k}-\lambda^{*}-Ex^{k} \|^{2}+(2+\epsilon_1)\rho \| Ex^{k}-Ex^{*} \|^{2} + \epsilon_2\rho \| Bx^{k}-Bx^{*} \|^{2}$. Clearly, $V^{k}$ is a nonnegative term. By combining (\ref{combining1}) and (\ref{combining2}), we obtain
		\begin{align}\label{combining}
			&V^{k}-V^{k+1} -\frac{\rho}{2} \left\| 2E\left( x^{k+1}-x^{k} \right)-Ax^{k+1} \right\| ^2 \nonumber\\
			&-\epsilon_1\rho \left\| E\left(x^{k+1}-x^{k}\right) \right\|^{2} - \epsilon_2\rho \left\| B\left(x^{k+1}-x^{k}\right) \right\|^{2}\geqslant 0 .
		\end{align}
		
		Then rearranging and summing preceding relation over $0, 1,..., s$ yields
		\begin{align}\label{0k}
			V^{0} - V^{s} \geqslant& \sum_{k=0}^s
			\left(\frac{\rho}{2}\left\|2E\left(x^{k+1}-x^{k}\right)-Ax^{k+1} \right\|^2 \right.\nonumber\\
			&\left. +\epsilon_1\rho \left\| E\left(x^{k+1}-x^{k}\right) \right\|^{2} + \epsilon_2\rho \left\| B\left(x^{k+1}-x^{k}\right) \right\|^{2}\right) \nonumber\\
			\geqslant& 0.
		\end{align}
		
		Inequality (\ref{0k}) implies that $V^{k}$ is bounded since $V_{0}$ is bounded and the left side of (\ref{0k}) is nonnegative. Thus the following equation holds for an arbitrary constant $\rho$, $\epsilon_1$ and $\epsilon_2$.
		\begin{align}\label{convergencetothesame}
			\lim _{k \rightarrow \infty}&
			\left(\frac{\rho}{2}\left\|2E\left(x^{k+1}-x^{k}\right)-Ax^{k+1} \right\|^2 \right.\nonumber\\
			&\left. +\epsilon_1\rho \left\| E\left(x^{k+1}-x^{k}\right) \right\|^{2} + \epsilon_2\rho \left\| B\left(x^{k+1}-x^{k}\right) \right\|^{2}\right) =0.
		\end{align}

		Eqn. (\ref{convergencetothesame}) implies that states of agents converge to the same point when the number of iterations goes to infinity. We now show that this point is the optimal point of the mentioned optimization Problem (\ref{problem4}).
		
		It is trivial to obtain that for $k \rightarrow \infty$, $x_{1}^{k+1}=x_{2}^{k+1}=\ldots=x_{n}^{k+1}$ and then $\lambda^{k+1}=\lambda^{k} $ from (\ref{MADMM2}) and (\ref{convergencetothesame}). 
		By summing up (\ref{vital}) over $1,2,...,n$, we obtain that
		\begin{align}\label{sumVI}
			&F\left(x\right)-F\left(x^{k+1}\right)+ \sum_{i=1}^{n}\left(x_{i}-x_{i}^{k+1}\right)^{\rm T} \nonumber\\
			&\cdot \Bigg\{-[A]_{i}\lambda 
			-\rho\left(\sum_{j \in P_{i}} x_{j}^{k}-|P_{i}| x_{i}^{k}\right)\Bigg\}
			\geqslant 0.
		\end{align}
		
		Notice that the optimal solution vector ${x}^*$ satisfies $Ax^{*}=0$, which implies the optimal solution $\bar{x}^*$ is the consensual value for all $x_{i}$, i.e., $(x_{1}^{*}, x_{2}^{*},..., x_{n}^{*})^{\rm T}= (\bar{x}^*,\bar{x}^*,...,\bar{x}^*)^{\rm T}$.
		For relation (\ref{sumVI}), let $x_{1}=x_{2}=...=x_{n}=\hat{x}$ and denote $\hat{u}=[\hat{x}\textbf{1}^{\rm T}, \lambda^{\rm T}]^{\rm T}\in \mathbb{R}^{n+m}$, $u^{k+1}=[{x^{k+1}}^{\rm T}, {\lambda^{k+1}}^{\rm T}]^{\rm T}\in \mathbb{R}^{n+m}$. Then for arbitrary $\hat{x}\in \mathbb{R}$, inequality (\ref{sumVI}) can be written as
		
		\begin{align}\label{39}
			&F\left(x\right)-F\left(x^{k+1}\right) + \left(\hat{u}-u^{k+1}\right)^{\rm T}\mathscr{F}\left(u^{k+1}\right)\nonumber\\ 
			&-\rho\left(\hat{x}-x^{k+1}\right)^{\rm T} \left(\sum_{i=1}^{n}|P_{i}|-\sum_{i=1}^{n}|P_{i}|\right)x \geqslant 0.
		\end{align}
		
		Thus the following relation is obtained
		\begin{align}\label{40}
			F\left(x\right)-F\left(x^{k+1}\right)+ \left(\hat{u}-u^{k+1}\right)^{\rm T}\mathscr{F}\left(u^{k+1}\right)\geqslant 0.
		\end{align}	
		
		Relation (\ref{40}) indicates that for a sufficiently large $k$, the vector $u^{k+1}$ is the solution to (\ref{VI}). Then by recalling Lemma \ref{lemma1}, we have that the solution of the preceding VI \emph{is equivalent to} the solution of our primary optimization problem (\ref{problem4}). 
	\end{proof}
	
	\begin{remark} 
		In \cite{wei2012distributed,yan2020parallel}, Only the illustration that the Lagrangian function or the global cost function constructed on the \emph{ergodic average} sequence $\{ \bar{x}^{s}=\frac{1}{s} \sum_{k=0}^{s-1}x^{k+1}\}$ converges to the optimal value is mentioned. However, the result that the state of each agent $\{x_i^k\}$ converges to the optimal state $\bar{x}^*$ is not developed. In this paper, by introducing the extra convex term $\epsilon_1\rho |P_{i}| \left\|x_{i}-x_{i}^{k} \right\|^{2} +\epsilon_2\rho |S_{i}| \left\|x_{i}-x_{i}^{k} \right\|^{2}$ and the variational inequality method, this issue is addressed, i.e., agents will reach a optimal consensual value according to \emph{Theorem 1}. 
	\end{remark} 
	
	The next theorem reveals the convergence rate for our algorithm.
	
	\begin{theorem}	
		Algorithm 2 converges at a rate of $O(\frac{1}{k})$. 
	\end{theorem}
	
	\begin{proof}	
		By adding (\ref{relation6})-(\ref{relation9}) into Lemma \ref{lemma3}, we obtain
		\begin{align}+
			&F\left(x^{*}\right)-F\left(x^{k+1}\right)\nonumber\\
			&-\frac{\rho}{2} \left\| 2E\left( x^{k+1}-x^{k}\right)-Ax^{k+1} \right\|^2 \nonumber\\
			&-\epsilon_1\rho \left\| E\left(x^{k+1}-x^{k}\right) \right\|^{2}-\epsilon_2\rho \left\| B\left(x^{k+1}-x^{k}\right) \right\|^{2}\nonumber\\
			&+\frac{1}{2\rho} \left( \left\| \lambda^{k}-\lambda^{*}-\rho Ex^{k} \right\|^2- 
			\left\| \lambda^{k+1}-\lambda^*-\rho Ex^{k+1} \right\|^2 \right) \nonumber\\
			&+\left(2+\epsilon_1\right)\rho\left(\left\| Ex^k-Ex^* \right\|^2 - \left\| Ex^{k+1}-Ex^{*} \right\|^2 \right)\nonumber\\
			&+ \epsilon_2\rho \left(\left\| Bx^{k}-Bx^{*} \right\|^{2}-\left\| Bx^{k+1}-Bx^{*} \right\|^{2}\right) 
			\geqslant 0.
		\end{align}
		
		Denote $\overline{V}^{k}=\frac{1}{2\rho} \| \lambda^{k}-Ex^{k} \|^{2}+(2+\epsilon_1)\rho \| Ex^{k}-Ex^{*} \|^{2} + \epsilon_2\rho \| Bx^{k}-Bx^{*} \|^{2}$, summing up from $0$ to $s-1$ yields
		\begin{align}\label{sum}
			&\overline{V}^{0} - \overline{V}^{s} +sF\left(x^{*}\right) - \sum_{k=0}^{s-1}F\left(x^{k+1}\right)\nonumber\\
			&-\sum_{k=0}^{s-1} \left(\frac{\rho}{2}\left\| 2E \left( x^{k+1}-x^{k} \right)-Ax^{k+1} \right\|^{2} \right. \nonumber\\
			&\left.+\epsilon_1\rho \left\| E\left(x^{k+1}-x^{k}\right) \right\|^{2} + \epsilon_2\rho \left\| B\left(x^{k+1} +x^{k}\right) \right\|^{2}\right) \geqslant 0.
		\end{align}
		
		Let $\overline{x}^{s}=\frac{1}{s} \sum_{k=0}^{s-1}x^{k+1}$.
		Since $F(x)$ is convex, we have $\sum_{k=0}^{s-1}F(x^{k+1}) \geqslant sF(\overline{x}^{k})$. Then we obtain 
		\begin{align}\label{44}
			sF\left(x^{*}\right)&-sF\left( \overline{x}^{s}\right) +\overline{V}^{0}\nonumber\\
			\geqslant& \overline{V}^{s} -\sum_{k=0}^{s-1} \left( \frac{\rho}{2}\left\| 2E \left( x^{k+1}-x^{k} \right)-Ax^{k+1} \right\|^{2} \right. \nonumber\\
			&\left.+\epsilon_1\rho \left\| E\left(x^{k+1}-x^{k}\right) \right\|^{2} + \epsilon_2\rho \left\| B\left(x^{k+1} +x^{k}\right) \right\|^{2}\right) \nonumber\\ 
			\geqslant& 0.
		\end{align}
		
		Finally by the re-arrangement of (\ref{44}), we obtain
		\begin{align}
			&F\left( \overline{x}^{s}\right)-F\left(x^{*}\right) \leqslant \frac{1}{s}\overline{V}^{0}.
		\end{align}
		
		We complete the proof.
	\end{proof}

	\section{Simulation}\label{section5}
	
	A distributed sensing problem is considered with a network of 9 nodes as shown in Fig. \ref{fig4}. For an unknown signal (using one-dimensional signal for convenience) $x\in \mathbb{R}$, each agent $i$ measures $x$ by using $y_i=M_ix+e_i$, where $y\in \mathbb{R}$ and $M_i\in \mathbb{R}$ are measured data following the standard Gaussian distribution $\mathcal{N}(0, 1)$. Then a consensus least squares problem is to be solved is as follows,
	\begin{flalign}
		\underset{x\in \mathbb{R}}{\operatorname{min}}
		\quad F\left(x\right)=\frac{1}{n} \sum_{i=1}^{n} \frac{1}{2} \left\| M_{i} x-y_{i} \right\|_{2}^{2}.
	\end{flalign}	
	
	We use $\frac{\|x^k- x^*\|_2}{\|x^*\|_2}$ to record the residual.
	Both prime and dual variables are initialized by zero. Penalty parameter $\rho$ is chosen to be 1 and constants $\epsilon_1$ and $\epsilon_2$ are 0. For comparison, two classical distributed optimization algorithms, Proximal Jacobian ADMM (PJADMM) \cite{deng2017parallel} and Distributed subgradient method (DSM) \cite{nedic2009distributed}, are also considered. PJADMM is an alternative distributed parallel ADMM-based algorithm, while DSM is one of the most well-known subgradient-based algorithms.
	\begin{figure}[H]
		\centering
		\includegraphics[width=3.5cm]{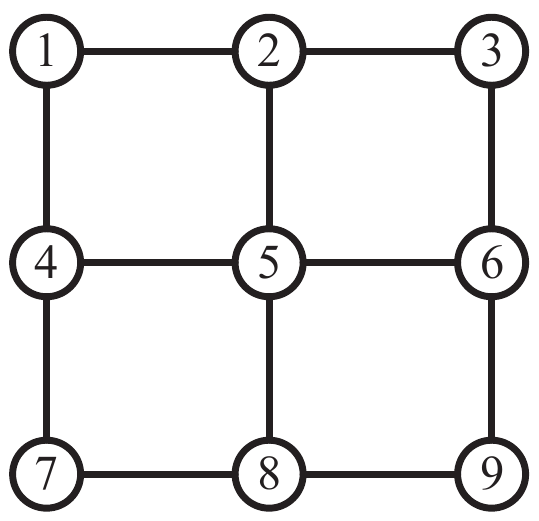}
		\centering
		\caption{A network of 9 nodes}
		\label{fig4}
	\end{figure}
	
	\subsection{PJADMM}
	We reformulate the typical form of the algorithm in \cite{deng2017parallel} as follows, 
	\begin{align}
		x_{i}^{k+1} = & \underset{x_{i}}{\operatorname{argmin}} \Bigg\{f_{i}\left(x_{i}\right)+\frac{1}{2}\left\|x_{i}-x_{i}^{k}\right\|^{2} \nonumber\\
		&+\sum_{j \in {N}_{i}}\left[\left(\lambda_{ij}^{k}\right)^{\rm T} x_{i}+\frac{\rho}{2}\left\|x_{i}-x_{j}^{k}\right\|^{2}\right]\Bigg\}, \nonumber\\
		\lambda_{ij}^{k+1} = & \lambda_{ij}^{k}-\rho\left(x_i^{k+1} - x_j^{k+1}\right),
	\end{align}	
	and the penalty parameter is set as $\rho = 1$ to match the fore-mentioned denotation.
	\subsection{DSM algorithm}
	A well-known algorithm named distributed subgradient method (DSM) \cite{boyd2010distributed} is shown as follows,
	\begin{align}
		x_{i}^{k+1}=\sum_{j=1}^{n} a_{j}^{i}\left(k\right) x_{j}^k-\alpha_{i}\left(k\right) d_{i}(x_i^k),
	\end{align}
	where the coefficients $a_{j}^{i}$ denote the weights. $\alpha_i(k)$ is the diminishing stepsize, which is chosen as $\alpha(k)=1/k^{1/2}$ in simulation. $d_i(x_i^k)$ is the subgradient of $f_{i}(x)$ at $x=x_i^k$.
	
	\subsection{Simulation Results}
	
	All the simulations are conducted in MATLAB R2018b with the aid of the convex optimization toolbox (CVX) developed by Stanford. 
	The simulation results are depicted in Figs. \ref{fig5}-\ref{fig8}.

	Fig. \ref{fig5} depicts that original Problem (\ref{problem1}) goes to the optimal value $F(x^*)=0.3014$ with Algorithm 2. The state updating records of $x_{i}^k, i=1, 2,..., 9$ are given in Fig. \ref{fig6}, which shows that the states reach a consensual optimal value at $x^*=-0.3890$. Fig. \ref{fig7} shows that our proposed algorithm achieves good performance and the convergence rate is greater than PJADMM and DSM algorithms. Fig. \ref{fig8} illustrates the difference by adjusting parameter $\epsilon = \epsilon_1 = \epsilon_2$ and shows that higher $\epsilon$ results in worse performance.
	\begin{figure}[H]
		\centering
		\includegraphics[width=5cm, height=4cm]{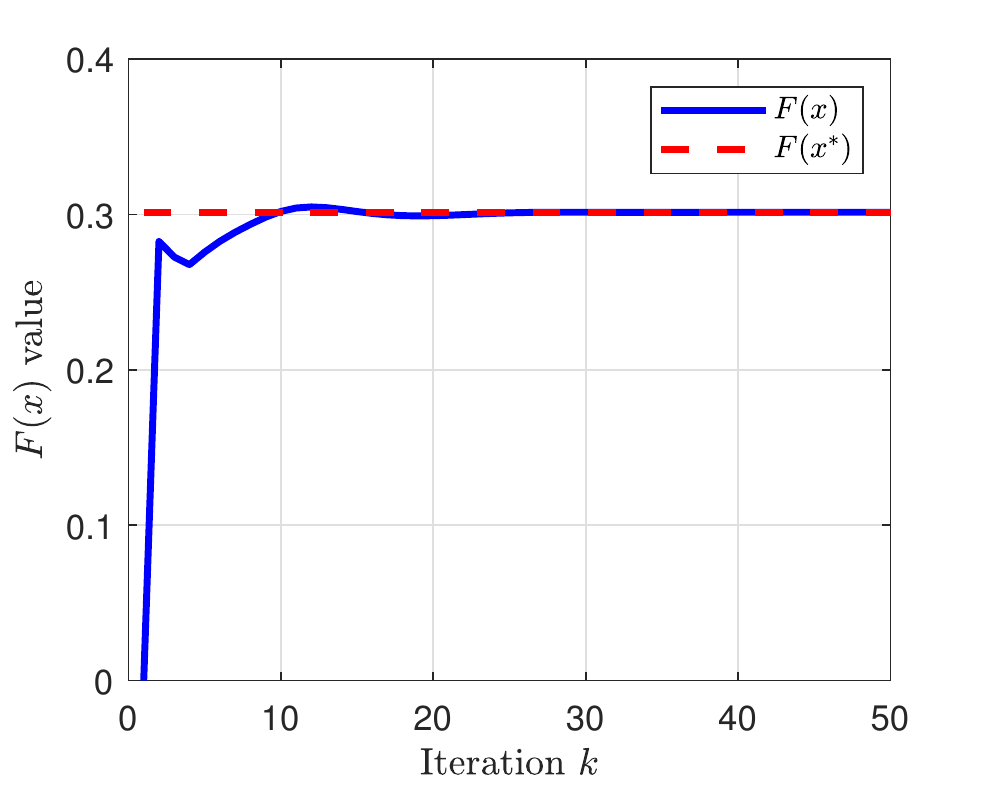}
		\centering
		\caption{Cost function value versus the number of iterations with the proposed method.}\label{fig5}
	\end{figure}
	
	\begin{figure}[H]
		\centering
		\includegraphics[width=5.0cm, height=4.0cm]{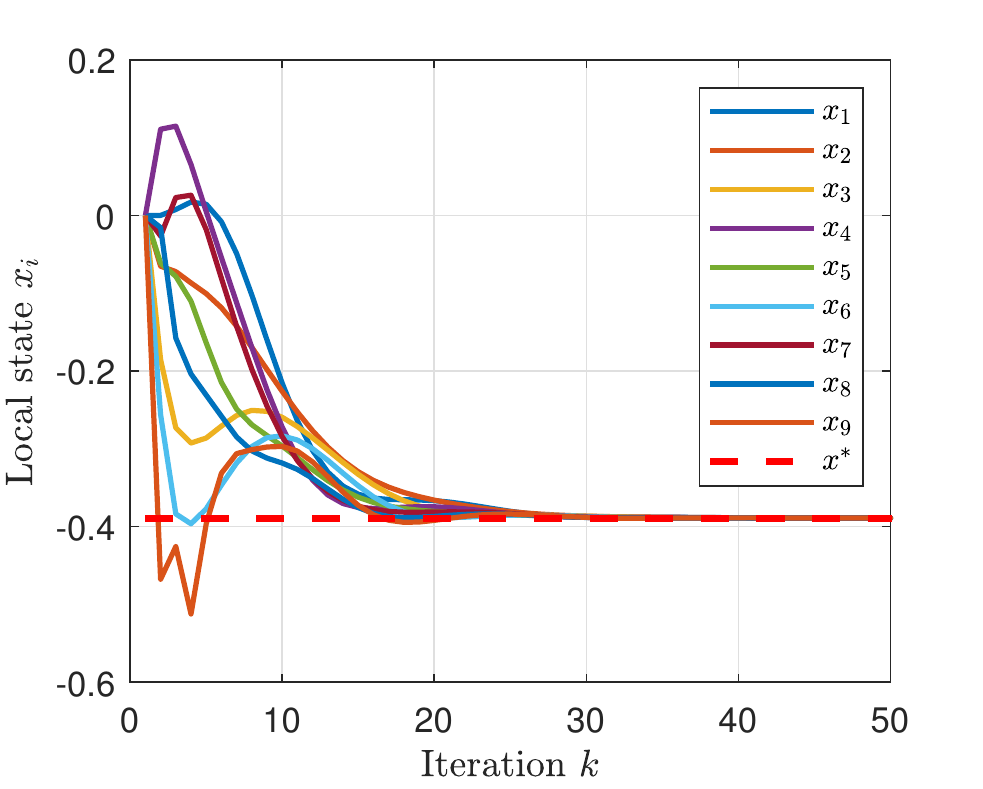}
		\centering
		\caption{Local states versus the number of iterations with the proposed method.}\label{fig6}
	\end{figure}
	
	\begin{figure}[H]
		\centering
		\includegraphics[width=5.0cm, height=4.0cm]{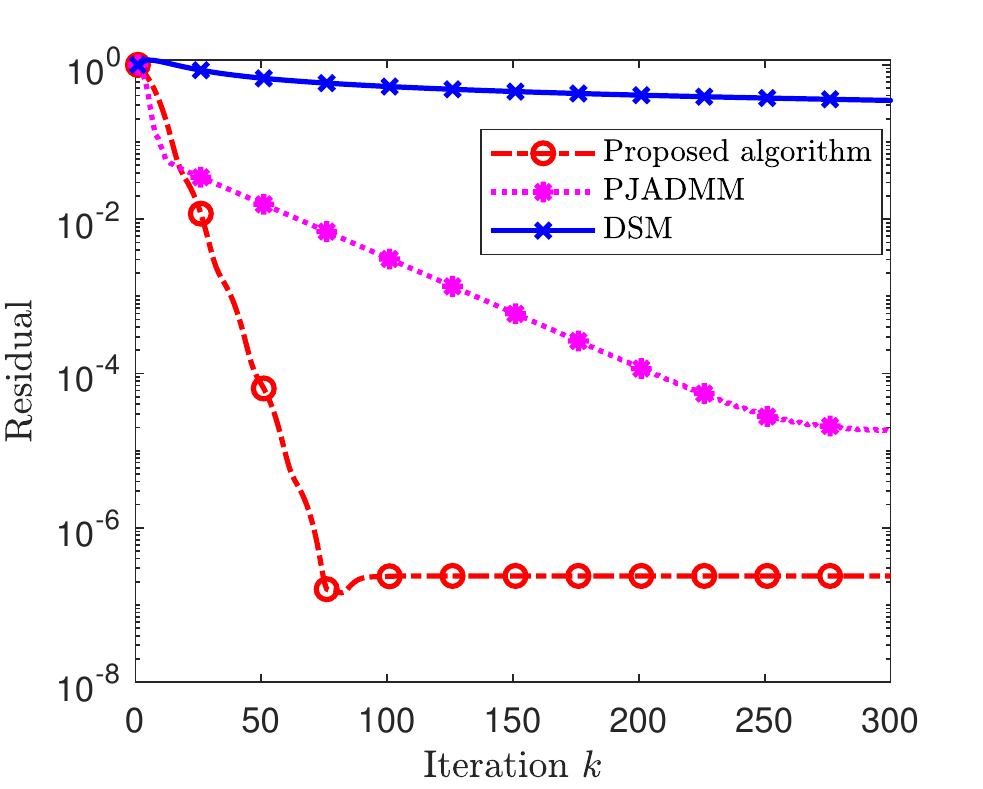}
		\centering
		\caption{Comparison with the existing algorithms.}\label{fig7}
	\end{figure}
	\begin{figure}[H]
		\centering
		\includegraphics[width=5.0cm, height=4.0cm]{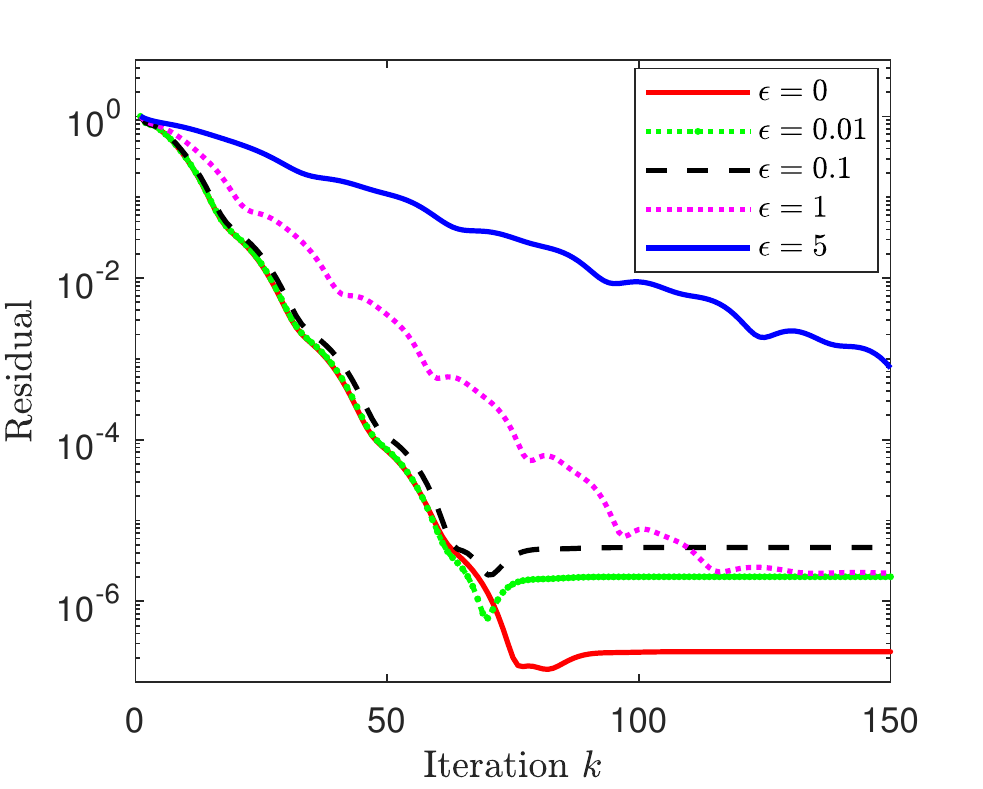}
		\centering
		\caption{Effect on the residual with different $\epsilon$.}\label{fig8}
	\end{figure}
	
	\section{Conclusion}\label{section6}
	In this paper, the distributed sequential ADMM algorithm in \cite{wei2012distributed} has been revised by remodeling the updating rule of primal variables, while without adding extra procedures. The newly presented distributed ADMM algorithm is suitable for parallel computing, and the optimality for both global cost function and the variables stored in agents has been proved. 
	A numerical example is considered to verify the effectiveness of our algorithm, which shows improved performance compared with a typical distributed parallel ADMM algorithm (JPADMM). 
	The issues of privacy-preservation and application of distributed optimization problem in smart grids will be interesting in future work. 
	
	%
	%
	%

	%

	\bibliographystyle{ieeetr} 
	\nocite{*} 
	\bibliography{main} 
\end{document}